\documentclass{amsproc}%
\usepackage{amsfonts}%
\usepackage{amsmath}%
\usepackage{amssymb}%
\usepackage{stmaryrd}%

\theoremstyle{plain}
\newtheorem{theorem}{Theorem}
\newtheorem{lemma}[theorem]{Lemma}
\newtheorem{corollary}[theorem]{Corollary}
\theoremstyle{remark}

\newtheorem{example}[theorem]{Example}
\newtheorem{remark}[theorem]{Remark}
\theoremstyle{definition}
\newtheorem{definition}[theorem]{Definition}

\newcommand{\C}{\mathbb{C}}
\newcommand{\E}{\mathbb{E}}
\newcommand{\N}{\mathbb{N}}
\newcommand{\bbp}{\mathbb{P}}
\newcommand{\R}{\mathbb{R}}
\newcommand{\T}{\mathbb{T}}
\newcommand{\Z}{\mathbb{Z}}

\newcommand{\cB}{\mathcal{B}}

\newcommand{\cF}{\mathcal{F}}
\newcommand{\cG}{\mathcal{G}}
\newcommand{\cH}{\mathcal{H}}

\newcommand{\cO}{\mathcal{O}}

\begin{document}
\title[L\'evy processes in compact groups]{\textbf{A limit theorem for occupation measures\\[1mm]
of L\'evy processes in compact groups}\\[6mm]}
\author{Arno Berger}
\address{Mathematical \& Statistical Sciences, University of Alberta, Edmonton, Alberta, Canada T6G 2G1 }
\email{aberger@math.ualberta.ca}
\urladdr{http://www.math.ualberta.ca/\symbol{126}aberger/}
\author{Steven N.\ Evans}
\address{Department of Statistics \#3860, 
367 Evans Hall, University of California, Berkeley, CA 94720-3860, USA}
\email{evans@stat.berkeley.edu}
\thanks{AB was supported by an NSERC Discovery Grant;
SNE was supported in part by NSF grant DMS-0907630.}
\urladdr{\texttt{http://www.stat.berkeley.edu/~evans/}}

\begin{abstract}
A short proof is given of a necessary and sufficient condition for the normalized occupation measure of
a L\'evy process in a metrizable compact group to be asymptotically uniform with probability one.
\end{abstract}

\maketitle

\section{Introduction}

Processes with stationary and independent increments are the continuous-time analogues of sums of independent, 
identically distributed random variables. They constitute one of the simplest yet also most fundamental classes 
of stochastic processes. With an additional assumption on the regularity of sample paths, such processes are 
referred to as {\em L\'evy processes}, see Definition~\ref{def:Levy} below. The class of L\'evy processes has 
been studied extensively and in many different state spaces, ranging from the classical case of $\R^d$ (where 
it contains both the Wiener and the Poisson process as extremely important examples) to more general topological 
groups, including non-Abelian ones. Due to their capability of incorporating both diffusion-type continuous 
evolution and jumps, L\'evy processes are now widely used as basic stochastic models in applied mathematics, 
notably in mathematical finance and quantum physics, see, for example, \cite{applebaum_04} which also contains 
ample references to the vast literature on the subject.

The aim of this note is to provide an easily accessible proof of a fundamental fact concerning the convergence 
to Haar measure of the normalized occupation measures of any L\'evy process taking values in a compact group. 
Various special cases of the main results, Theorem \ref{thm1} and Corollary \ref{cor4}, have been (re-)discovered 
repeatedly over the years, as have some related facts, see Remark \ref{rem8}. However, arguments geared towards 
these special cases tend to obscure rather than elucidate the underlying general principle. As this note ventures 
to demonstrate, the latter is most easily understood when stripped of all superfluous particulars.

\section{Basic definitions and notations}

Throughout, $G$ denotes a metrizable compact group, with the group operation written multiplicatively, with 
neutral element $e_G$ and with Borel $\sigma$-algebra $\cB_G$. When written without a subscript, the symbol 
$\cB$ stands for the Borel $\sigma$-algebra on $\R$, or on some (Borel) subset thereof. The sets $gB$ and $Bg$ 
are the images of $B\in \cB_G$ under, respectively, the left- and right-translation by $g$; that is, 
$gB = \{gb: b \in B\}$ and $Bg = \{ bg: b\in B\}$. Write $\lambda_G$ for the (normalized) Haar measure on 
$G$; that is, $\lambda_G$ is the unique probability measure on $(G, \cB_G)$ that is invariant under all 
left-translations. (Equivalently, $\lambda_G$ is the unique probability measure on $(G, \cB_G)$ invariant 
under all right-translations.) For any closed (and hence compact) subgroup $H$ of $G$ it will be understood 
that the corresponding Haar measure $\lambda_H$ is defined on all of $(G, \cB_G)$, rather than merely on 
$(H, \cB_H)$, and $\lambda_H(G\backslash H)=0$. For any $g \in G$, denote by $\epsilon_g$ the Dirac 
probability measure concentrated at $g$.

Denote by $C(G)$ the separable Banach space of continuous, complex-valued functions on $G$, equipped with the 
supremum norm. The dual of $C(G)$ is the space of finite, complex-valued measures on $(G, \cB_G)$; from now 
on this space will always be equipped with the corresponding weak$^*$ topology.

\begin{definition}
A measurable function $\gamma:[0,+\infty) \to G$ is {\em continuously uniformly distributed\/} in $G$, 
abbreviated henceforth as {\em c.u.d.}, if
\begin{equation}\label{eq0}
\lim\nolimits_{T\to +\infty} \frac1{T} \int_0^T \varphi \bigl(  \gamma(t) \bigr)\, 
{\mathrm d}t = \int_G \varphi \, {\mathrm d}\lambda_G 
\quad
\forall \varphi \in C(G) \, .
\end{equation}
Similarly, with $\lfloor y \rfloor$ denoting, as usual, the largest integer not larger
than $y\in \R$, a sequence $(g_n)_{n\in \N}$ is {\em uniformly distributed\/}
({\em u.d.}) in $G$ whenever the function $\gamma: t\mapsto g_{\lfloor t \rfloor + 1}$ is c.u.d., see \cite{KN}.
\end{definition}

Note that $\gamma$ is c.u.d.\ if and only if the normalized occupation measures $\Lambda_T$ converge to 
$\lambda_G$ as $T\to +\infty$; here the probability measure $\Lambda_T$ is, for every $T>0$, defined by  
$$
\Lambda_T(B) = \frac{1}{T} \int_0^T {\bf 1}_B \bigl(\gamma(t)\bigr) 
\, {\rm d}t \quad \forall B \in \cB_G \, ,
$$
with ${\bf 1}_B$ denoting the indicator function of any set $B \in \cB_G$.

\begin{definition}
\label{def:Levy}
A {\em L\'evy process} in $G$ is a family $X=(X_t)_{t\ge 0}$ of $G$-valued random variables, defined on some 
underlying probability space $(\Omega, \cF, \bbp)$, with the following properties:

\smallskip

\begin{enumerate}
\item For any $0\le t_1 < t_2$ the distribution of the increment $X_{t_1}^{-1}X_{t_2}$ is
the same as the distribution of $X_0^{-1}X_{t_2 - t_1}$.

\smallskip

\item The random variables $X_{t_1}, X_{t_1}^{-1}X_{t_2}$, $X_{t_2}^{-1}X_{t_3}$, \dots,
$X_{t_{n-1}}^{-1}X_{t_n}$ are independent whenever $n\ge 2$ and $0 \le t_1 < t_2 < \ldots < t_n$.

\smallskip

\item For $\bbp$-almost all $\omega \in \Omega$, the $G$-valued function $t \mapsto X_t(\omega)$ is right-continuous 
with left-limits (or {\em rcll\/} for short); that is, $\lim_{\varepsilon \downarrow 0} X_{t+\varepsilon}(\omega) = X_t(\omega)$ for all 
$t \ge 0$, and $\lim_{\varepsilon \downarrow 0} X_{t - \varepsilon}(\omega) =: X_{t-}(\omega)$ exists for all $t>0$.
\end{enumerate}
\end{definition}

\noindent
For fixed $\omega \in \Omega$, the rcll function $t\mapsto X_t(\omega)$ is referred to as a {\em path\/} of $X$.  

Write $D$ for the set of all rcll functions from $[0,\infty)$ to $G$. Note that the equivalent French acronym 
{\em c\`adl\`ag} is often used instead of rcll, and $D$ is called the {\em Skorohod space\/} associated with $G$. 
It is possible to equip $D$ with a complete, separable metric such that the corresponding Borel $\sigma$-algebra 
$\cB_D$ coincides with the $\sigma$-algebra generated by the sets of the form 
\[
 \{ \gamma \! \in \!  D : \gamma (t_j) \! \in \! B_j \:\: \mbox{\rm for }
n \! \in \! \N ; \; j=1,\ldots , n ; \; 0 \le t_1 <  \ldots < t_n ; \;
B_1,  \ldots , B_n \! \in \! \cB_G
 \}  \,  ,
\]
see, for example, \cite[Sections 3.5 and 3.7]{Ethier_Kurtz_86}.

\section{Main result and applications}

Let $X=(X_t)_{t\ge 0}$ be a L\'evy process in $G$. For every $t\ge 0$, write $\mu_t$ for the distribution 
of the increment $X_0^{-1} X_t$. Note that the family of probability measures $(\mu_t)_{t \ge 0}$ is a 
{\em convolution semigroup}; that is, $\mu_{t_1} \ast \mu_{t_2} = \mu_{t_1 + t_2}$ for all $t_1,t_2 \ge 0$, 
where $\ast$ denotes convolution of probability measures on $(G, \cB_G)$. It follows that each 
probability measure $\mu_t$  is {\em infinitely divisible}, though no explicit use of this property 
will be made here. For any probability measure $\nu$ on $(G, \cB_G)$, recall that the {\em support\/} 
of $\nu$ is the smallest closed set $F\subseteq G$ with  $\nu(F)=1$. For every $t\ge 0$ write $S_t$ 
for the support of $\mu_t$. The following characterization of the almost sure continuous uniform 
distribution for the paths of $X$ is the main content of this note.

\begin{theorem}\label{thm1}
Let $X=(X_t)_{t\ge 0}$ be a L\'evy process in the metrizable compact group $G$. Then the following
statements are equivalent:

\begin{enumerate}

\item 
The set $\bigcup_{t\ge 0} S_t$ is dense in $G$.

\smallskip

\item The paths of $X$ are, with probability one, c.u.d.\ in $G$.

\end{enumerate}
\end{theorem}

\begin{proof}
In order not to interrupt the main thread, two
auxiliary facts of a technical nature are deferred to the subsequent Lemmas \ref{lemmadd1} and \ref{lemmadd2}.

It will be convenient to formulate the main part of the proof using the terminology of ergodic theory. To this end, 
consider the probability measure $\rho$  defined on $(G \times D, \cB_G \otimes \cB_D)$ by setting, for any
$n\in \N$, $0= t_0 < t_1  < \ldots < t_n$ and $B_0, B_1, \ldots , B_n \in \cB_G$,
$$
\rho \bigl(B_0 \times \{ \gamma \! \in \! D : \gamma (t_j) \! \in \! B_j \: \: \forall j = 1, \ldots , n \}\bigr) :=
\bbp \{ \xi \! \in \! B_0,  X_0^{-1}X_{t_j} \! \! \in \! B_j \: \: \forall  j=1,\ldots , n\} \, ,
$$
where $\xi$ is a random variable, also defined on $(\Omega, \cF , \bbp)$, that is independent of the process 
$X$ and has distribution $\lambda_G$. For every $t\ge 0$, define a map $R_t$ of $G \times D$ into itself by
$$
R_t \bigl(g,\gamma (\bullet) \bigr) = \bigr( g \gamma(0)^{-1} \gamma(t), \gamma(t)^{-1} \gamma (t+\bullet )\bigr) 
\quad \forall (g, \gamma) \in G \times D \, .
$$
Clearly, $R_t$ is measurable, and
\[
\begin{split}
 R_{t_1} & \! \circ \!  R_{t_2} \bigl(g, \gamma (\bullet) \bigr)  = R_{t_1} \left( g \gamma(0)^{-1}\gamma(t_2), 
 \gamma(t_2)^{-1} \gamma (t_2 + \bullet )\right)  \\
& = \Bigl( g \gamma(0)^{-1}\gamma(t_2) \gamma(t_2)^{-1} \gamma (t_2 + t_1 ), \left(\gamma(t_2)^{-1}  
\gamma (t_2 + t_1)\right)^{-1} \!\! \gamma(t_2)^{-1}  \gamma (t_2 + t_1 + \bullet )  \Bigr) \\
& = \Bigl( g \gamma(0)^{-1} \gamma(t_1 + t_2) , \gamma(t_1 + t_2)^{-1} \gamma (t_1 + t_2 + \bullet )\Bigr) \\
& = R_{t_1 + t_2} \bigl( g, \gamma(\bullet) \bigr) \\
\end{split}
\]
holds for all $t_1,t_2\ge 0$. Moreover, since $\xi X_0^{-1}X_t$ has distribution $\lambda_G$ for all 
$t\ge 0$, it follows from the stationarity and independence of increments of $X$ that for any $n\in \N$, $
0=t_0 < t_1 < \ldots < t_n$ and $B_0, B_1, \ldots , B_n \in \cB_G$,
\[
\begin{split}
\rho \! \circ \! R_t^{-1} \! \bigl(B_0  \! & \times \! \{ \gamma : \gamma (t_j) \in B_j \: \: \forall j = 1,  \ldots , n \}\bigr) = \\
& = \bbp \{\xi X_0^{-1} \! X_{t} \in B_0, \; X_t^{-1} \! X_{t+t_j} \in B_j \: \: \forall  j= 1,\ldots , n\} \\
& = \bbp \{\xi \in B_0, \;  X_0^{-1} \! X_{t_j} \in B_j \: \: \forall j= 1,\ldots , n\} \\
& = \rho \bigl(B_0 \! \times \!  \{ \gamma : \gamma (t_j) \in B_j \: \: \forall j = 1, \ldots , n \}\bigr) \, .\\
\end{split}
\]
Thus, $(R_t)_{t\ge 0}$ is a $\rho$-preserving semi-flow. In particular, the stochastic process $(\xi X_0^{-1} X_t)_{t\ge 0}$ 
is {\em stationary} \cite{CFS, krengel}.
Recall that $(R_t)_{t\ge 0}$ is said to be {\em ergodic\/} if $\rho (R_t^{-1} (A) \triangle  A)=0$ 
for $A \in \cB_G \otimes \cB_D$ and all $t\ge 0$ implies that $\rho (A)\in \{0,1 \}$, where, as usual, $\triangle$ 
denotes the symmetric difference of two sets. From Lemma \ref{lemmadd1} below, it follows 
that $(R_t)_{t\ge 0}$ is ergodic if and only if
\begin{equation}\label{eqadd1}
\bbp \bigl\{ \lambda_G (BX_0^{-1} X_t \triangle B) = 0 \bigr\} =1 \quad \forall t\ge 0 
\end{equation}
for some set $B\in \cB_G$ implies that $\lambda_G(B)\in \{0,1\}$. (Note that Lemma \ref{lemmadd1} is required only for the 
``if'' part; the ``only if'' part is straightforward, cf.\ \cite[Theorem 3]{kakutani_50} and 
\cite[Theorem 1]{Ryll-Nardzewski_54}.)


With these preparations, the asserted implication (i)$\Rightarrow$(ii) will now be proved. 
Thus assume (i); that is, assume $\overline{\bigcup_{t\ge 0} S_t}=G$.  
The key step in establishing
(ii) is to check that the semi-flow $(R_t)_{t\ge 0}$ is ergodic in this case. Suppose, therefore, that (\ref{eqadd1}) holds for some set $B\in \cB_G$.
Note that then
$$
\bbp  \bigl\{ \lambda_G (BX_0^{-1} X_{t_n} \triangle B) = 0 \: \: \forall n \in \N \bigr\} =1
$$
holds for every sequence $(t_n)_{n\in \N}$ in $[0, +\infty)$. Specifically, choose $(t_n)_{n\in \N}$ 
such that  $\bigcup_{n\in \N} S_{t_n}$ is dense in $G$. (This is possible due to the separability of $G$.) 
By Fubini's Theorem, $\lambda_G (B h \triangle B)=0$ holds for all $h$ in a dense subset $H_0$ of $G$. 
To establish the ergodicity of $(R_t)_{t\ge 0}$, it remains to demonstrate how this last conclusion 
implies that $\lambda_G (B) \in \{0,1\}$. Assume, therefore, that $\lambda_G(B)>0$ and let $\nu$ be the 
normalized restriction of $\lambda_G$ to $B$; that is,
$$
\nu (A):= \frac{\lambda_G (B \cap A)}{\lambda_G (B)} \quad \forall A \in \cB_G \, .
$$
For $g\in G$, denote by $T_g$ the right-translation by $g$.  Notice that, for every $h\in H_0$,
\[
\nu \circ T_h^{-1} (A) = \frac{\lambda_G (B \cap Ah^{-1})}{\lambda_G(B)} = \frac{\lambda_G (Bh \cap A)}{\lambda_G (B)} 
 = \nu (A) + \frac{\lambda_G (Bh \cap A) - \lambda_G (B\cap A)}{\lambda_G (B)} \, , 
\]
and hence
$$
\left| 
\nu \circ T_h^{-1} (A) - \nu (A)
\right| =
\frac{|\lambda_G(Bh \cap A) - \lambda_G (B\cap A) |}{\lambda_G (B)} \le \frac{\lambda_G (Bh \triangle B)}{\lambda_G (B)} = 0 \, ,
$$
showing that $\nu \circ T_h^{-1} = \nu$. Given any $g \in G$ and $\varphi \in C(G)$, pick a sequence 
$(h_n)_{n\in \N}$ in $H_0$ such that $\lim_{n\to \infty} h_n = g$. Since $\varphi \circ T_{h_n} \to \varphi \circ T_g$ 
uniformly on $G$, it follows by dominated convergence that
\begin{align*}
\int_G \varphi (x) \, {\mathrm d} \nu \circ T_g^{-1}(x) & = \int_G \varphi (xg) \, {\mathrm d} \nu (x) 
= \lim \nolimits_{n\to \infty} \int_G \varphi (xh_n) \, {\mathrm d}\nu (x) \\
& =
\lim \nolimits_{n\to \infty} \int_G \varphi  (x) \, {\mathrm d}\nu (x) = \int_G \varphi (x) \, {\mathrm d} \nu(x) \, . 
\end{align*}
Thus, $\nu$ is invariant under {\em all\/} right-translations, and consequently $\nu = \lambda_G$. In particular, $\lambda_G(B)=1$,
as required. Hence, the semi-flow $(R_t)_{t\ge 0}$ is ergodic.

By the Birkhoff Ergodic Theorem, for every integrable $\varphi : G \to \C$,
$$
\frac1{T} \int_0^T \varphi \bigl( g \gamma(0)^{-1} \gamma(t) \bigr)\, {\mathrm d}t 
\: \stackrel{T \to + \infty}{ \longrightarrow } \: \int_{G \times D} \varphi (g') \, {\mathrm d} \rho (g', \gamma') =
\int_G \varphi (g') \, {\mathrm d} \lambda_G(g') 
$$
holds for $\rho$-a.e. $(g,\gamma)\in G\times D$. In probabilistic terms, this means that 
\begin{equation}\label{eqadd2}
\lim\nolimits_{T\to +\infty} \frac1{T} \int_0^T \varphi (\xi X_0^{-1} X_t) \, {\mathrm d}t =
\int_G \varphi \, {\mathrm d} \lambda_G
\end{equation}
holds with probability one. For any $\varphi \in C(G)$, $g\in G$ and $n\in \N$, denote the set
$$
\left\{
\omega \in \Omega : \limsup\nolimits_{T\to +\infty} \left| 
\frac1{T} \int_0^T \varphi \bigl( g  X_t(\omega) \bigr)\, {\mathrm d}t - \int_G \varphi
\, {\mathrm d}\lambda_G
\right| < \frac1{n}
\right\} \: \in \: \cF
$$
by $\Omega_{\varphi, g, n}$. As $\xi X_0^{-1}$ is Haar-distributed in $G$, $1 = \int_G \bbp (\Omega_{\varphi, g, n})\, {\mathrm d}\lambda_G(g)$ 
for every $n$ by the above, and so $\bbp (\Omega_{\varphi, g, n}) = 1$ for $\lambda_G$-almost every $g\in G$. If $\lambda_G (\{e_G \})>0$ or, 
equivalently, if $G$ is finite, then $\bbp (\Omega_{\varphi, e_{_G} , n}) = 1$ for all $n$, and consequently 
$\bbp \left(\bigcap_{n} \Omega_{\varphi, e_{_G} , n}\right)=1$. If, on the other hand, $e_G$ is not an atom 
of $\lambda_G$ then, by the uniform continuity of $\varphi$, there exists a sequence $(g_n)_{n\in \N}$ in $G$ 
with $\lim_{n\to \infty} g_n = e_G$ such that $\bbp (\Omega_{\varphi, g_n, 2n}) = 1$ and 
$\Omega_{\varphi, g_n, 2n} \subseteq \Omega_{\varphi , e_{_G} , n}$ for all $n$. From
$$
1 = \bbp \left( 
\bigcap\nolimits_{n} \Omega_{\varphi, g_n , 2n} 
\right)
\le
\bbp \left( 
\bigcap\nolimits_{n} \Omega_{\varphi, e_{_G},n} 
\right)
\le 1 \, ,
$$
it is clear that also in this case 
\begin{equation}\label{eqadd3}
\lim\nolimits_{T\to +\infty} \frac1{T} \int_0^T \varphi ( X_t) \, {\mathrm d}t =
\int_G \varphi \, {\mathrm d} \lambda_G
\quad
\mbox{\rm with probability one.}
\end{equation}
Finally, recall that $C(G)$ is separable, and hence taking the intersection of (\ref{eqadd3}) over a dense family 
$\{\varphi_n : n \in \N\}$ in $C(G)$ yields that the paths $t\mapsto X_t$ are, with probability one, c.u.d.\ in $G$. 
Thus, (i) implies (ii).


To show the reverse implication (ii)$\Rightarrow$(i), suppose that (i) does not hold. In this
case, Lemma \ref{lemmadd2} shows that $H_X:= \overline{\bigcup_{t\ge 0} S_t}$ is a proper (compact) subgroup of $G$. 
It then follows from the first part of the proof that the paths of $(X_0^{-1} X_t)_{t \ge 0}$ are c.u.d. in $H_X$. Thus, 
\begin{equation}\label{eqpf3a}
\lim\nolimits_{T\to +\infty} \frac1{T} \int_0^T \varphi \bigl( X_t \bigr)\, 
{\mathrm d}t = \int_{H_X} \varphi(X_0 h) \, {\mathrm d}\lambda_{H_X}(h) \quad \forall \varphi \in C(G)
\end{equation}
holds with probability one. It is straightforward to see that, no matter what the distribution of $X_0$ is, for 
some suitable choice of $\varphi \in C(G)$ the integral on the right of (\ref{eqpf3a}) will not almost surely equal 
$\int_G \varphi(g) \, {\mathrm d}\lambda_G(g)$.
\end{proof}

The following two somewhat technical lemmas have been relied on in the proof of Theorem \ref{thm1}.

\begin{lemma}\label{lemmadd1}
With the notation used in the proof of Theorem \ref{thm1}, let the $\cB_G \otimes \cB_D$-measurable function 
$\psi : G \times D \to \R$ be invariant under the semi-flow $(R_t)_{t\ge0}$; that is, assume that 
$\psi \circ R_t = \psi$ holds $\rho$-a.e.\ for all $t\ge 0$. Then, there exists a $\cB_G$-measurable 
function $\overline{\psi}:G\to \C$ such that $\psi (g,\gamma) = \overline{\psi} (g)$ for $\rho$-a.e.\ 
$(g,\gamma)\in G\times D$.
\end{lemma}

\begin{proof}
The argument given here mimics the proof of \cite[Theorem 1]{Ryll-Nardzewski_54}.
Write expectations and conditional expectations with respect to $\rho$ as $\rho[\bullet]$ and $\rho[\bullet \, | \, \bullet]$,
respectively. Put $\mathcal{G} := \cB_G \otimes \{\emptyset, D\}$ and $\mathcal{H} := \{\emptyset,G\} \otimes \cB_D$.
Since there exists, for any bounded $\cB_G \otimes \cB_D$-measurable 
function $\psi$, a $\cB_G$-measurable function $\overline \psi: G \to \mathbb{R}$
such that $\rho[\psi \, | \, \mathcal{G}](g,\gamma) = \overline \psi(g)$ for $\rho$-a.e.\ $(g,\gamma) \in G \times D$, 
it suffices to show for any bounded invariant function $\psi$ that $\psi = \rho[\psi \, | \, \mathcal{G}]$ $\rho$-a.e..
By a monotone class argument, 
this is equivalent to showing for any bounded invariant
function $\psi$, any bounded $\mathcal{G}$-measurable function $\alpha$,
and any bounded $\mathcal{H}$-measurable function $\beta$ that
\begin{equation}
\label{exp_psi_alpha_beta}
\rho\left[\psi \, \alpha \, \beta\right] 
= 
\rho \bigl[\rho \left[\psi \, | \, \mathcal{G} \right] \, \alpha \, \beta \bigr] \, .
\end{equation}
Note that due to the independence of the sub-$\sigma$-algebras $\mathcal{G}$
and $\mathcal{H}$ under $\rho$ the right-hand side of
\eqref{exp_psi_alpha_beta} is 
\[
\rho \bigl[\rho \left[\psi \, | \, \mathcal{G} \right] \, \alpha \bigr] \; \rho\left[\beta \right]
=
\rho \left[\psi \, \alpha \right] \; \rho\left[\beta \right]\, ,
\]
and so it further suffices to show for any bounded $\mathcal{H}$-measurable
function $\beta$ with $\rho[\beta]=0$ that $\rho\left[\psi \, \alpha \, \beta\right] = 0$
for any bounded $\mathcal{G}$-measurable function $\alpha$. Since
\[
\rho\left[\psi \, \alpha \, \beta\right]
= \rho\bigl[\rho\left[\psi \, \beta \, | \, \mathcal{G}\right] 
\, \alpha \bigr],
\]
this is equivalent to establishing
\begin{equation}
\label{exp_psi_beta_zero}
\rho\bigl[
\big| \rho\left[\psi \, \beta \, | \, \mathcal{G}\right] \big| 
\bigr] = 0 \, .
\end{equation}
Note also that $\rho\bigl[\big| \rho\left[\psi \, \beta \, | \, \mathcal{G}\right] \big| \bigr]
= \rho\bigl[ \big| \rho \left[\psi \, \beta \, | \, \mathcal{G} \right] \circ R_t \big|\bigr]$
for all $t \ge 0$ because $R_t$ preserves the measure $\rho$. Moreover, observe that
\[
\rho\left[\psi \, \beta \, | \, \mathcal{G}\right](g,\gamma)
=
\mathbb{E} \left[\psi(g, X_0^{-1} X_{\bullet}) \,
\overline \beta(X_0^{-1} X_{\bullet})\right]\, ,
\]
where $\overline \beta$ is a $\cB_D$-measurable function such that $\beta(g,\gamma) = \overline \beta(g)$ 
for $\rho$-a.e.\ $(g,\gamma) \in B \times D$.  By definition of $R_t$ and the stationary increments property of $X$,
\[
\begin{split}
\rho \left[\psi \, \beta \, | \, \mathcal{G} \right] \circ R_t(g, \gamma) 
& =
\mathbb{E} \left[\psi(g \gamma(0)^{-1} \gamma(t), X_0^{-1} X_{\bullet}) 
\, \overline \beta(X_0^{-1} X_{\bullet})\right] \\
& =
\mathbb{E} \left[\psi(g \gamma(0)^{-1} \gamma(t), X_t^{-1} X_{t+\bullet}) 
\, \overline \beta(X_t^{-1} X_{t+\bullet})\right]\, . \\
\end{split}
\]
Since $\psi$ is invariant, $\psi(g,\gamma) = \psi\bigl( g \gamma(0)^{-1} \gamma(t), \gamma(t)^{-1} \gamma (t+\bullet )\bigr)$, 
and hence
\[
\begin{split}
\rho \left[\psi \, \beta \, | \, \mathcal{G} \right] \circ R_t(g, \gamma) 
=
\mathbb{E} \left[\psi(g, X_\bullet^{(\gamma,t)})
\, \overline \beta(X_t^{-1} X_{t+\bullet})\right] \, , \\
\end{split}
\]
where
\[
X_s^{(\gamma,t)} := 
\begin{cases}
\gamma(s) & \mbox{\rm if } 0 \le s < t \, ,\\
\gamma(t) X_t^{-1} X_{s} & \mbox{\rm if } s \ge t \, .\\
\end{cases}
\]
For every $t \ge 0$, let $\mathcal{H}_t$ be the sub-$\sigma$-algebra of $\cB_G \otimes \cB_D$ generated by
the maps $(g,\gamma) \mapsto \gamma(s)$, $0 \le s \le t$, and denote, as usual, by $\cG \vee \cH_t$ the $\sigma$-algebra
generated by $\cG \cup \cH_t$. Since $\rho \left[\psi \, | \, 
\mathcal{G} \vee \mathcal{H}_t\right]
(g, X_\bullet^{(\gamma,t)}) = \rho \left[\psi \, | \, 
\mathcal{G} \vee \mathcal{H}_t\right]
(g, \widetilde \gamma)$ for any $\widetilde \gamma \in D$ such that $\widetilde \gamma(s) = \gamma(s)$ for $0 \le s \le t$,
it follows that
\[
\begin{split}
\mathbb{E} \left[\rho \left[\psi \, | \, 
\mathcal{G} \vee \mathcal{H}_t\right]
(g, X_\bullet^{(\gamma,t)}) \,
\overline \beta(X_t^{-1} X_{t+\bullet}) \right]
& =
\rho \left[\psi \, | \, 
\mathcal{G} \vee \mathcal{H}_t\right]
(g, \gamma)
\mathbb{E} \left[\overline \beta(X_t^{-1} X_{t+\bullet}) \right] \\
& =
\rho \left[\psi \, | \, 
\mathcal{G} \vee \mathcal{H}_t\right]
(g, \gamma) \, \rho[\beta] \\
&= 0 \, . \\
\end{split}
\]
Also, by the independent increments property of $X$ and the Martingale Convergence Theorem,
\[
\begin{split}
& \int_{G \times D}
\mathbb{E} \left[
\left|
\rho \left[\psi \, | \, 
\mathcal{G} \vee \mathcal{H}_t\right]
(g, X_\bullet^{(\gamma,t)})
-
\psi(g, X_\bullet^{(\gamma,t)})
\right|
\right]
\, \mathrm{d} \rho(g,\gamma) \\
& \quad =
\int_{G \times D}
\left|
\rho \left[\psi \, | \, 
\mathcal{G} \vee \mathcal{H}_t\right]
(g, \gamma)
-
\psi(g, \gamma)
\right|
\, \mathrm{d} \rho(g,\gamma)
\: \stackrel{t \to \infty}{\longrightarrow} \: 0 \, . \\
\end{split}
\]
Hence,
\begin{align*}
\rho \bigl[ \big| \rho  [ & \psi \, \beta \,  | \, \mathcal{G} ] \big| \bigr]  =
\rho \bigl[ \big| \rho \left[\psi \, \beta  \,   | \, \mathcal{G} \right] \circ R_t \big| \bigr] 
 = \int_{G \times D} \!
\left| 
\mathbb{E} \left[\psi(g, X_\bullet^{(\gamma,t)}) \, \overline \beta(X_t^{-1} X_{t+\bullet})\right]
\right|
\, \mathrm{d}\rho(g,\gamma) \\
& =
\int_{G \times D}
\biggl|
\mathbb{E} \left[\psi(g, X_\bullet^{(\gamma,t)})
\, \overline \beta(X_t^{-1} X_{t+\bullet})\right] \\
& \qquad \qquad -
\mathbb{E} \left[\rho \left[\psi \, | \, 
\mathcal{G} \vee \mathcal{H}_t\right]
(g, X_\bullet^{(\gamma,t)}) \,
\overline \beta(X_t^{-1} X_{t+\bullet}) \right]
\biggr|
\, \mathrm{d}\rho(g,\gamma) \\
& \le
\|\beta\|_\infty
\int_{G \times D}
\mathbb{E} \left[ \left|
\psi(g, X_\bullet^{(\gamma,t)}) - \rho \left[\psi \, | \, \mathcal{G} \vee \mathcal{H}_t\right]
(g, X_\bullet^{(\gamma,t)})
\right|
\right]
\, \mathrm{d} \rho(g,\gamma) 
\:  \stackrel{t \to \infty}{\longrightarrow} \: 0 \, , 
\end{align*}
which in turn shows that \eqref{exp_psi_beta_zero} holds.  
Therefore, for each invariant function $\psi$ 
there does indeed exist a $\cB_G$-measurable function $\overline \psi: G \to \mathbb{R}$
such that $\psi(g,\gamma) = \overline \psi(g)$ 
for $\rho$-a.e. $(g,\gamma) \in G \times D$.
\end{proof}

\begin{lemma}\label{lemmadd2}
For every L\'evy process in the metrizable compact group $G$, the set $\overline{\bigcup_{t\ge 0} S_t}$ is a subgroup of $G$.
\end{lemma}

\begin{proof}
For convenience, denote $\overline{ \bigcup_{t\ge 0} S_t }$ by $H_X$. 
Note that $\mu_{t_1} \ast \mu_{t_2} = \mu_{t_1 + t_2}$ implies
$$
S_{t_1} S_{t_2} := \{ g_1 g_2 : g_j \in S_{t_j} \: \mbox{\rm for } j=1,2 \} = S_{t_1 + t_2} \quad \forall t_1, t_2 \ge 0 \, .
$$
It follows that $h_1h_2\in H_X$ whenever $\{h_1, h_2\}\subset H_X$. Given any $h\in H_X$, therefore,
$h H_X \subseteq H_X$. To see that $h H_X = H_X$, first choose a metric $d$ on $G$ that is invariant 
under all left- as well as right-translations. (Such a metric exists, see, for example, \cite[\S0.6]{walters_82}.)
If $g \in H_X \setminus h H_X$, then $d(g, h^n g) \ge \min\{d(g, h h') : h' \in H_X\} > 0$ for every $n\in \N$, 
and in particular $d(h^m g, h^n g) = d(g, h^{n-m} g)$ is bounded away from zero for $m,n \in \N$, $n>m$. 
Consequently, the sequence $(h^{n}g)_{n\in \N}$ does not contain any convergent subsequence, contradicting
the compactness of $G$. Hence $hH_X = H_X$, and it is clear that $\{e_G, h^{-1}\}\subset H_X$. 
Since $h\in H_X$ was arbitrary, the set $H_X$ is indeed a subgroup.
\end{proof}

If $(X_t)_{t\ge 0}$ is a L\'evy process with $X_0 = e_G$ and $h>0$, then $X_{nh}$ is, for every $n\in \N$, the 
product of $n$ independent random variables, all of which have the same distribution as $X_h$. Conversely, let 
$(\xi_n)_{n\in \N}$ be an i.i.d.\ sequence in the metrizable compact group $G$. Denote by $S$ the support of
the common distribution of the $\xi_n$, $n\ge 1$, and, for every $n\in \N$, let $S^n = \{ g_1 \cdots g_n : g_j \in S \;\,
\mbox{\rm for } j=1,\ldots, n\}$. If $(\tau_n)_{n \in \N}$ is a sequence of i.i.d.\ exponential random variables
that is independent of $(\xi_n)_{n\in \N}$, then the process $(X_t)_{t\ge 0}$ defined by $X_t = e_G$ for 
$0 \le t < \tau_1$ and $X_t = \xi_1  \cdots \xi_n$ for $\tau_1 + \cdots + \tau_n \le t < \tau_1 + \cdots + \tau_{n+1}$
is a L\'evy process. The following discrete-time analogue is, by the latter observation, immediate from
Theorem \ref{thm1} and the Strong Law of Large Numbers applied to the random variables $(\tau_n)_{n \in \N}$.

\begin{corollary}\label{cor4}
Let $(\xi_n)_{n\in \N}$ be an i.i.d.\ sequence in the metrizable compact group $G$ with the common 
distribution having support $S$. Then, the sequence $(\xi_1 \cdots \xi_n)_{n\in \N}$ is, with probability one, 
u.d.\ in $G$ if and only if $\overline{\bigcup_{n \in \N} S^n} = G$.
\end{corollary}

\begin{example}\label{exa5}
Let $Y= (Y_t)_{t\ge 0}$ be a L\'evy process in the (non-compact, Abelian) group $\R$ with the usual 
topology. According to the classical L\'evy--Khintchine formula \cite[Theorem 1.2.14]{applebaum_09}, 
$\E [e^{iy Y_t}]= e^{t\eta (y)}$ for all $t\ge 0$ and $y\in \R$, where
\begin{equation}\label{eqlk}
\eta (y) = i \beta y - {\textstyle \frac12} \sigma^2 y^2 + \int_{\R} \left( e^{ixy} - 1 - i xy {\bf 1}_{(-1,1)}(x)\right) {\mathrm d}\nu (x) \, ,
\end{equation} 
with $\beta \in \R$, $\sigma^2 \ge 0$, and $\nu$ a Borel measure on $\R$ that satisfies
$\nu (\{ 0 \})=0$ and $\int_{\R} y^2 \wedge 1  \, {\mathrm d}\nu (y) < + \infty$. The triple 
$(\beta ,\sigma^2 , \nu)$ uniquely determines $Y$.

For any $y\in \R$, denote by $\langle y \rangle$ the fractional part of $y$, that is, 
$\langle y \rangle = y - \lfloor y \rfloor$. Set $X_t = \langle Y_t \rangle$ for all $t\ge 0$. 
Clearly, $X$ is a L\'evy process in the compact (Abelian) group $\T=\R/\Z$. From (\ref{eqlk}) 
and Theorem \ref{thm1} it is readily deduced that the paths of $X$ are, with probability one, 
c.u.d.\ in $\T$ unless simultaneously $\sigma^2 = 0$ (i.e., $Y$ has no Gaussian component), 
$\nu ( \R \setminus \frac1{m} \Z) = 0$ for some $m\in \N$ (i.e., $\nu$ is concentrated on 
the lattice $\frac1{m}\Z = \{ \frac{k}{m} : k \in \Z\}$), and $\beta = \sum_{|k|< m} \frac{k}{m}\nu (\{ \frac{k}{m}\})$. 
In the latter case, assuming that $m$ is chosen to be minimal, the paths of $X$ are c.u.d.\
in the closed subgroup $\langle \frac1{m} \Z \rangle$ of $\T$, that is, with probability one,
$$
\lim \nolimits_{T\to +\infty} \frac1{T} \int_0^T \varphi( X_t) \, {\mathrm d}t =
\frac{1}{m}
\sum\nolimits_{j=0}^{m-1} \varphi \left( \textstyle \frac{j}{m}\right) \quad
\forall \varphi \in C (\T) \, .
$$

Note that with $\vartheta \in \R$ and $\bbp \{\xi_1 = \vartheta\}=1$, Corollary \ref{cor4} contains 
the well known fact that $\bigl(\langle n\vartheta\rangle \bigr)_{n\in \N}$ is u.d.\ in $\T$ if and 
only if $\vartheta$ is irrational.
\end{example}

\begin{example}\label{exa5a}
If $G$ is merely {\em locally\/} compact then (\ref{eq0}) will usually not hold for almost all 
paths of a L\'evy process $X$ in $G$, even when both sides exist, perhaps for some appropriate
subspace of $C(G)$. However, Example~\ref{exa5}  can be extended in a way that not only 
highlights the role played by the compactness of $G$, but also provides a new perspective on 
\cite[Theorem 1]{robbins_53}. 

Let $Y$ again be a L\'evy process in $\R$, with characteristic triple $(\beta, \sigma^2, \nu)$, 
and fix a bounded continuous function $f:\R \to \C$. To avoid trivialities, assume $f$ is non-constant. 
Theorem \ref{thm1} can be used to show that $\lim\nolimits_{T\to + \infty}  \frac1{T} \int_0^T f(Y_t) \, {\mathrm d}t$ 
does exists with probability one, provided that $f$ is {\em almost periodic\/} (or {\em a.p.} for short). Recall that 
$f$ is a.p.\ if, for every $\varepsilon > 0$, there exists a set $P_{\varepsilon}\subseteq \R$ which is
relatively dense (i.e., $P_{\varepsilon}$ ``has bounded gaps'') such that
$$
\sup\nolimits_{y\in \R} |f(y+p) - f(y)| < \varepsilon \quad \forall p\in P_{\varepsilon} \, .
$$
It is well known that $f$ is a.p.\ if and only if the closure $H_f$ of the family $\{ f(y + \bullet) : y \in \R\}$ 
is compact in $C_b (\R)$, the Banach space of bounded continuous complex-valued functions on $\R$ equipped with
the supremum norm. (Usually, $H_f$ is referred to as the {\em hull\/} of $f$, see, for example, \cite{fink}.)  Moreover,
the average
$$
A(f) := \lim\nolimits_{T \to + \infty} \frac1{2T} \int_{-T}^T f(y) \, {\mathrm d} y
$$
exists for every a.p.\ function $f$. The addition 
$$
f(y_1 + \bullet) + f (y_2 + \bullet) := f(y_1 + y_2 + \bullet) \quad \forall y_1, y_2 \in \R \, ,
$$
extends continuously to $H_f$, turning the latter into a metrizable compact (Abelian) group. Clearly,
$e_{H_f}= f$, and the Haar measure on $H_f$ is uniquely determined by the requirement that
$$
\int_{H_f} \varphi \, {\mathrm d} \lambda_{H_f} = A(\varphi \diamond f) \quad \forall \varphi \in C(H_f) \, ,
$$
where $\varphi \diamond f$ denotes the a.p.\ function $y\mapsto \varphi \bigl(  f(y+\bullet )\bigr)$.
With these preparations, define a process $X$ in $H_f$ by simply setting $X_t = f(Y_t + \bullet)$ for all $t\ge 0$.
It is readily confirmed that $X$ is a L\'evy process. (Note that Example \ref{exa5} simply
corresponds to the special case of $f$ being periodic with period $1$, in which case $H_f$ is homeomorphic
and isomorphic to $\T$, and $A(f) = \int_0^1 f(y) \, {\mathrm d}y$.) Observe that $H_X=H_f$ unless
$\sigma^2 =0$ and $\nu (\R \backslash a \Z)=0$ for some $a>0$. When $H_X=H_f$, Theorem \ref{thm1} 
implies that, with probability one,
$$
\frac1{T} \int_0^T  \!\! \varphi (X_t) \, {\mathrm d} t = \frac1{T} \int_0^T \!\! \varphi \bigl(  f (Y_t + \bullet) \bigr)\, {\mathrm d} t \:
\stackrel{T\to +\infty}{\longrightarrow} \int_{H_f} \!\! \varphi \, {\mathrm d} \lambda_{H_f} = A(\varphi \diamond f) \quad \forall \varphi \! \in \! C(H_f) \, .
$$
In particular, choosing $\varphi (g):= g(0)$ for all $g \in H_f$ yields $\varphi \diamond f = f$ and consequently
\begin{equation}\label{eqap}
\frac1{T} \int_0^T f(Y_t) \, {\mathrm d}t \: \stackrel{T\to +\infty}{\longrightarrow} \: A(f) \quad
\mbox{\rm with probability one} \, .
\end{equation}
For every a.p.\ function $f$, therefore, (\ref{eqap}) holds for any L\'evy process $Y$ on $\R$ provided that
$Y$ either has a non-zero Gaussian component or else the associated measure $\nu$ is not concentrated on a lattice.
\end{example}

\begin{example}\label{exa7}
As an application of Theorem~\ref{thm1} and Corollary~\ref{cor4}, let $b\ge 2$ be a positive
integer and recall that a measurable function $f:[0,+\infty) \to \R$ is $b$-{\em Benford\/} if
$\log_b |f|$ is c.u.d.\ in $\T$, where $\log_b$ denotes the base-$b$ logarithm and the
convention $\log_b 0:=0$ is adopted for convenience, see \cite{BH} for background
information and further details on the Benford property as well as its ramifications.
Equivalently, the function $f$ is $b$-Benford if
$$
\lim\nolimits_{T\to +\infty} \frac{\mbox{\rm Leb} \bigl\{ t \in [0,T) : 
S_b\bigl( f( t )\bigr) \le s \bigr\} }{T} = \log_b s \quad \forall s \in [1,b) \, ,
$$
where $S_b(y)$, the base-$b$ {\em significand\/} of $y\in \R$ is, by definition, the unique
number in $\{0\} \cup [1,b)$ such that $|y| = S_b(y) b^k$ for some integer $k$.  Similarly, 
a sequence $(y_n)_{n\in \N}$ in $\R$ is called $b$-Benford whenever the function 
$t\mapsto y_{\lfloor t  \rfloor + 1}$ is $b$-Benford.

Let $Y$ be a L\'evy process in $\R$ with characteristic triple $(\beta , \sigma^2 , \nu)$ and,
for any real constants $a\ne 0$, $c\ne 0$, and $d$, consider 
$$
X_t = a e^{cY_t + dt} \, , \quad t \ge 0 \, .
$$
Since $(cY_t + dt)_{t\ge 0}$ is again a L\'evy process, it follows from Theorem~\ref{thm1}
that $t\mapsto X_t$ is $b$-Benford with probability one unless
\begin{equation}\label{eq3}
\sigma^2 = 0  \quad \mbox{\rm and} \quad \nu \left( \R \backslash {\textstyle \frac{\ln b}{|c|m}} \Z\right) = 0
\enspace
\mbox{for some } m\in \N \, .
\end{equation}
Note that (\ref{eq3}) does not hold if $\nu$ is non-atomic. In particular, the paths of any 
{\em geometric Brownian motion\/} (also referred to as a {\em Black--Scholes process\/}), 
corresponding to the case where $Y$ is a standard Brownian motion, are almost surely $b$-Benford 
for all $b$. Similarly, if $Y$ is a Poisson process then the paths $t\mapsto a e^{cY_t + dt}$ are, 
with probability one, $b$-Benford unless $c$ is a rational multiple of $\ln b$, cf.\ \cite{schuerger_08}.

For a discrete-time analogue of these observations, let $(\xi_n)_{n\in \N}$ be an i.i.d.\ sequence in $\R$. 
By Corollary \ref{cor4}, the sequence $\left( \prod_{j=1}^n \xi_j \right)_{n\in \N}$ is $b$-Benford with
probability zero or one, depending on whether or not the support of the distribution of $\:\! (\log_b |\xi_1|)$
is contained in $\frac1{m} \Z$ for some $m\in \N$, cf.\ \cite{ross_10}.
\end{example}

\section{Further remarks and observations}

The following remarks aim at providing some background information that may help the reader putting the main 
results of this note, Theorem \ref{thm1} and Corollary \ref{cor4}, in perspective.

\begin{remark}
The notion of continuous uniform distribution applies to all measurable functions (paths), not only to
those that are rcll. Hence it may seem natural to consider the class of stochastic processes that arises
by replacing assumption (iii) in Definition~\ref{def:Levy} with the weaker requirement that the map 
$(t,\omega) \mapsto X_t(\omega)$ be jointly measurable. As the following argument shows, nothing is
gained from this seemingly greater generality.

By \cite[Theorem 1]{dr_02}, joint measurability implies the existence of closed sets $F \subseteq [0,1]$
with Lebesgue measure arbitrarily close to $1$ such that for all $\varepsilon > 0$
\[
\lim\nolimits_{\delta \downarrow 0} \sup\nolimits_{|t_2-t_1| \le \delta, \, t_1,t_2\in F} \bbp\{d(X_{t_1}, X_{t_2}) 
> \varepsilon\} = 0 \, ,
\]
where $d$ is a translation-invariant metric on $G$. Observe that if $t_2 > t_1$, then
\[
\bbp\{d(X_{t_1}, X_{t_2}) > \varepsilon\} 
 = \bbp\{d(e_G, X_{t_1}^{-1} X_{t_2}) > \varepsilon\} 
 = \mu_{t_2-t_1}\bigl( \{g \in G : d(e_G,g) > \varepsilon\}\bigr) \, . 
\]
A celebrated theorem of Steinhaus \cite[Th\'eor\`eme VIII]{Steinhaus_20} (see also 
\cite{ Bingham_Ostaszewski_11, Kestelman_47}) asserts that the set $\{t_2-t_1 : t_1,t_2 \in F\}$ contains 
an open neighbourhood of $0$ whenever $F$ has positive Lebesgue measure. Thus, $\mu_t$
converges to $\epsilon_{e_G}$ as $t \downarrow 0$.  Now, define a strongly continuous contraction
semigroup of operators $(P_t)_{t \ge 0}$ on $C(G)$ by setting $P_0\varphi = \varphi$ and
$$
P_t \varphi = \int_G \varphi (\bullet \, g ) \, {\mathrm d}\mu_t(g) \quad
\forall t>0 \,\; \mbox{\rm and } \varphi \in C(G) \, .
$$
It follows from the theory of Feller semigroups (see, for example, \cite[Section III.7]{Rogers_Williams_94})
that by modifying each random variable $X_t$ on a $\bbp$-null set it is possible to produce a stochastic
process with rcll paths. By Fubini's Theorem, the value of $\int_0^T \varphi(X_t) \, {\mathrm d}t$, 
$T \ge 0$, remains unchanged if $X$ is replaced by such an rcll modification.
\end{remark}

\begin{remark}
\label{rem:Harris}
L\'evy processes are a special class of Markov processes, and so it is natural to inquire whether 
Theorem~\ref{thm1} is a consequence of more general results in the vast Markov process literature. The 
proof given above certainly uses the extra L\'evy structure: The fact that the distribution of 
$(X_t)_{t \ge 0}$ when $X_0 = g$ is equal to the distribution of $(g X_t)_{t \ge 0}$ when $X_0 = e_G$
reduces checking the ergodicity of 
$(\xi X_0^{-1} X_t)_{t \ge 0}$ to verifying a criterion involving only the right-translations
by $h\in H_X$. Moreover, the fact that the state space is a compact group permits the latter 
verification to be reduced to the uniqueness of normalized Haar measure on such a group.

A discussion of limit theorems for the occupation measures of discrete-time Markov processes
is given in \cite[Chapter 17]{Meyn_Tweedie_93} under the assumption that the process is 
{\em Harris recurrent}.  Similar results for continuous-time
processes are obtained in \cite[Paragraphe II]{AKR_67} by the device
of sampling the process at the arrival times of a Poisson
process to obtain a discrete-time process. 
The condition $H_X=G$ in Theorem~\ref{thm1} is easily seen to be 
equivalent to the condition that $\int_0^{+\infty} \!\! \mu_t(U) \, {\mathrm d}t > 0$ for all 
non-empty open sets $U \subseteq G$.  If this condition is replaced by the stronger assumption
that $\int_0^{+\infty} \!\! \mu_t(B) \, {\mathrm d}t > 0$ for all $B\in \cB_G$ with $\lambda_G(B) > 0$, 
then it is possible to conclude from the results in 
\cite{AKR_67, Meyn_Tweedie_93} that
$\lim_{T\to +\infty} \frac1{T} \int_0^T \varphi ( X_t) \, {\mathrm d}t =
\int_G \varphi \, {\mathrm d} \lambda_G$ almost surely for any bounded measurable function
$\varphi$.  Note that if $Y= (Y_t)_{t\ge 0}$ is as in Example~\ref{exa5} with $(\beta , \sigma^2 , \nu)
= (0,0, \epsilon_{\vartheta})$ for some irrational $\vartheta \in \R$, then $X = (X_t)_{t \ge 0}$ 
defined by $X_t = \langle Y_t \rangle$ satisfies the condition $H_X=\T$, yet
$\int_0^\infty \bbp\{X_t \in B\} \, {\mathrm d}t = 0$ when $B$ is the complement
of $\{\langle n \vartheta \rangle : n \in \N\} \cup \{0\}$ in $\T$, a set with full 
$\lambda_{\T}$-measure.
\end{remark}

\begin{remark}\label{rem4a}
The ergodicity of the stationary process $(\xi X_0^{-1}X_t)_{t\ge 0}$ appearing in the proof of Theorem \ref{thm1}
and that of its discrete-time analogue $(\xi \xi_1  \cdots \xi_n)_{n\in \N}$ can be established more easily
if one assumes that, respectively, $S_t$ for some $t > 0$ and the support $S$ of the common distribution of 
the $\xi_n$, $n\ge 1$, are not contained in the coset of any proper closed normal subgroup of $G$.
Under this additional assumption, it follows from the It\^{o}--Kawada Theorem \cite[Theorem 2.1.4]{heyer_77}
that, for any $t\ge 0$, the random variables $X_t^{-1}X_{t+T}$ converge in distribution
to $\lambda_G$ as $T\to +\infty$, and analogously, $\xi_n \xi_{n+1} \cdots \xi_{n+N}$
converges in distribution to $\lambda_{G}$ as $N\to \infty$. As a consequence, for any $A,B \in \cB_G$,
$$
\bbp \{\xi X_0^{-1} X_t \in A , \xi X_0^{-1} X_{t+T }\in B \}  \: \stackrel{T\to +\infty}{\longrightarrow} \: 
\bbp \{\xi X_0^{-1} X_t \in A\} \, \bbp \{\xi X_0^{-1} X_t \in B \}
\, ,
$$
showing that the process $(\xi X_0^{-1} X_t)_{t\ge 0}$ is actually {\em mixing\/} in this case, and thus 
{\em a fortiori\/} ergodic, cf.\ \cite{krengel, walters_82}. Mixing properties of the semi-flow 
$(R_t)_{t\ge 0}$ have been studied in \cite{georgii_97}. By contrast, the proof of Theorem \ref{thm1} 
presented above only uses the ergodicity of $(\xi X_0^{-1} X_t)_{t\ge 0}$. Ergodicity is all that
can be hoped for in general: For example, take $G = \T$ and let $(\xi_n)_{n\in \N}$ be an i.i.d.\ 
sequence with $\bbp \bigl\{\xi_1 = \langle \sqrt{2} \rangle\bigr\}=1$. In this case, the
process $(\xi + \xi_1 + \ldots + \xi_n)= \bigl(\langle \xi + n \sqrt{2} \rangle \bigr)$, though
stationary and ergodic, is {\em not\/} mixing, since for $B = \{ \langle t \rangle : 0 \le t \le \frac12 \}\in \cB_{\T}$, 
$$
\limsup \nolimits_{N\to \infty} \bbp \bigl\{  \langle \xi + n \sqrt{2} \rangle  \! \in \! B ,
\langle \xi + (n+N) \sqrt{2} \rangle \! \in \!  B \bigr\} = {\textstyle \frac12 \ne \frac14 } =
\bbp \{\langle \xi + n \sqrt{2} \rangle \! \in \!  B\}^2
$$
holds for every $n\in \N$.
\end{remark}

\begin{remark}
Every Hausdorff compact group $G$ carries a unique normalized Haar measure. Hence the statement of
Theorem \ref{thm1} (resp.\ Corollary \ref{cor4}) makes sense for every jointly measurable stochastic process $X$ (resp.\ 
sequence) in $G$. Even though it makes sense, however, it is not generally true: The asserted equivalence
may break down whenever the increments of $X$ are non-stationary or dependent, or if $G$ fails to be
metrizable. While the former two observations are quite obvious, to see what might
go wrong when $G$ is not metrizable, recall that $G$ is metrizable
if and only if $C(G)$ is separable. The proof of Theorem \ref{thm1} given here uses the metrizability of $G$,
the separability of $C(G)$, and the consequent separability of $G$, and
so there is no hope that this proof will extend.
It is shown in \cite[Corollary 4.5.4]{KN} that if
a Hausdorff compact Abelian group $G$ is not separable, then no sequence is u.d.\ in $G$. Any 
extension of Corollary \ref{cor4}, therefore, must require the separability of $G$
(at least in the Abelian case). 
\end{remark}

\begin{remark}\label{rem8}
As the following short, non-exhaustive compilation illustrates, special cases of Theorem \ref{thm1} 
and Corollary \ref{cor4} as well as related results have repeatedly appeared in the literature. 

The earliest pertinent references the authors have been able to identify are the announcement
of a {\em Random Ergodic Theorem\/} in \cite{uvn} and its subsequent significant generalization 
in \cite{kakutani_50, Ryll-Nardzewski_54}. In a purely probabilistic setting, \cite{robbins_53} focuses on
discrete-time processes taking values in $\R$ but also considers the case $G=\T$. In addition, extensions to compact
groups and general continuous-time processes $X=(X_t)_{t\ge0}$ in $\R$ are discussed briefly. For the
latter, a sufficient condition for the almost sure continuous uniform distribution of paths is
given under the assumption that
\begin{equation}\label{eq10}
\E \bigl[ e^{i\lambda (X_t - X_0)} \bigr] = \cO (t^{-\delta}) \quad \mbox{as } t\to +\infty
\end{equation}
for every real $\lambda \ne 0$ and the appropriate $\delta = \delta (\lambda)>0$. Note that (\ref{eq10}) 
holds for every non-degenerate Brownian motion in $\R$, but it does not hold if $X$ is, for instance, a 
Poisson process, since in this case $\left|\E \left[ e^{i\lambda (X_t - X_0)}\right] \right| = 1$ whenever 
$\lambda  \in 2\pi \Z$. Theorem \ref{thm1} replaces (\ref{eq10}) with a necessary and sufficient condition.

The uniform distribution of Brownian paths on $\R$ has been established in \cite{derman_54} and
subsequently in \cite{hlawka_70}. Building on these, in the case of Brownian motion on 
$\R$, \cite{stack_71} proves a law of the iterated logarithm for the deviations of
$\frac1{T} \int_0^T \varphi (X_t) \, {\mathrm d}t$ from its expected value, and \cite{BDT}
study the same problem on compact connected Riemannian manifolds, while \cite{loynes_73,schatte_85} 
consider more general processes on $\R$. A sufficient condition for sequences of real-valued random 
variables with stationary, but not necessarily independent increments to be u.d.\ in $\T$ is derived 
in \cite{holewijn_69}, and in \cite{stadje_89} sums of i.i.d.\ random variables are considered under 
the perspective of rotation invariance.

It appears that the Benford property for paths of (some) L\'evy processes has been studied
only rather recently. Utilizing large deviation results, \cite{schuerger_08} essentially
establishes the almost sure c.u.d.\ property of the paths of $X$ for $G = \T$ with 
$X = \langle Y \rangle$, where $Y$ is a continuous local martingale plus a deterministic 
drift. The most important example of this type is standard Brownian motion, and the test 
function $\varphi$ in (\ref{eq0}) may be taken to be merely {\em measurable\/} and 
{\em bounded\/} in this case, i.e., $\varphi \in L^{\infty} (G)$ instead of $\varphi \in C(G)$, 
cf.\ Remark \ref{rem:Harris}. (Notice that $L^{\infty}(G)$ is non-separable whenever $G$ is 
infinite.) In \cite{schuerger_11}, a similar approach is extended to general L\'evy processes 
in $\R$. In this more general setting, however, the desired conclusion -- the almost sure 
Benford property of paths -- is obtained only under an additional regularity condition on the
characteristic function of $Y_1$, referred to as ``standard condition''. Many L\'evy
processes, most importantly perhaps any Poisson process, do not satisfy this condition
and hence are not amenable to the techniques of \cite{schuerger_11}. Although this may look
like a minor technicality, it is not: As shown in Remark~\ref{rem:Harris}, there are
L\'evy processes $X$ on $\T$ for which $\lim_{T\to +\infty} \frac1{T}\int_0^T \varphi (X_t) \, {\mathrm d}t=
\int_{\T} \varphi \, {\mathrm d}\lambda_{\T}$ does not almost surely hold for some $\varphi \in L^{\infty}(\T)$.
\end{remark}

\subsection*{Acknowledgements}

The first author is much indebted to 
T.\ Hill, V.\ Losert, V.\ Runde and B.\ Schmuland for
many stimulating conversations and helpful suggestions. He also wishes
to acknowledge that, prior to this note and independently of it, the
Benford property for L\'evy processes has been studied in the
as yet unpublished treatise \cite{schuerger_11}.


\vspace*{-0.5mm}

\begin{thebibliography}{BNT}

\bibitem{applebaum_04} D.\ Applebaum, L\'evy Processes -- From Probability
to Finance and Quantum Groups, {\it Notices Amer.\ Math.\ Soc.\/} {\bf 51}
(2004), 1336--1347.

\bibitem{applebaum_09} D.\ Applebaum, {\em L\'evy Processes and Stochastic Calculus\/} (2nd ed.),
Cambridge University Press, Cambridge, 2009.

\bibitem{AKR_67} J.\ Az\'ema, M.\ Kaplan-Duflo and D.\ Revuz, Mesure invariante sur les
classes r\'ecurrentes des processus de Markov, {\em Z.\ Wahrscheinlichkeitstheorie und Verw.\ Gebiete\/}
{\bf 8} (1967), 157--181.

\bibitem{BH} A.\ Berger and T.P.\ Hill, A Basic Theory of Benford's Law, {\it Probab.\ Surv.\/} {\bf 8} (2011), 1--126.

\bibitem{Bingham_Ostaszewski_11}
N.H.\ Bingham and A.J.\ Ostaszewski, Dichotomy and infinite combinatorics: 
the theorems of {S}teinhaus and {O}strowski, {\it Math.\ Proc.\ Cambridge  Philos.\ Soc.}
{\bf 150} (2011), 1--22.

\bibitem{BDT} M.\ Bl\"{u}mlinger, M.\ Drmota and R.F.\ Tichy, A uniform
law of the iterated logarithm for Brownian motion on compact Riemannian
manifolds, {\it Math.\ Z.\/} {\bf 201} (1989), 495--507.

\bibitem{CFS} I.P.\ Cornfeld, S.V.\ Fomin and Ya.G.\ Sinai, {\it Ergodic Theory},
Springer, Berlin--Heidelberg--New York, 1982.

\bibitem{derman_54} C.\ Derman, Ergodic property of the Brownian motion
process, {\it Proc.\ Nat.\ Acad.\ Sci.\ USA\/} {\bf 40} (1954), 1155--1158.

\bibitem{dr_02} G.\ Di Nunno and Yu.A.\ Rozanov, On measurable modification of stochastic functions,
{\it Theory Probab.\ Appl.} {\bf 46} (2002), 122--127.

\bibitem{Ethier_Kurtz_86}
S.N.\ Ethier and T.G.\ Kurtz, {\it Markov processes}, John Wiley \& Sons, New York, 1986.
     
\bibitem{fink} A.M.\ Fink, {\em Almost Periodic Differential Equations}, Lecture Notes in Mathematics {\bf 377},
Springer, Berlin--Heidelberg--New York, 1974.

\bibitem{georgii_97} H.-O.\ Georgii, Mixing properties of induced random transformations,
{\it Ergodic Theory Dynam.\ Systems\/}  {\bf 17} (1997), 839--847.

\bibitem{Hewitt_Ross_79}
E.\ Hewitt and K.A.\ Ross, {\it Abstract harmonic analysis. Vol. I} (2nd edn.),
Grundlehren der Mathematischen Wissenschaften [Fundamental Principles of Mathematical Sciences]
{\bf 115}, Springer, Berlin, 1979.

\bibitem{heyer_77} H.\ Heyer, {\em Probability measures on locally compact groups},
Springer, Berlin--New York, 1977.

\bibitem{hlawka_70} E.\ Hlawka, Ein metrischer Satz in der Theorie der $C$-Gleichverteilung,
{\em Monatsh.\ Math.} {\bf 74} (1970), 108--118.

\bibitem{holewijn_69} P.J.\ Holewijn, On the uniform distribution of sequences of random
variables, {\it Z.\ Wahr\-scheinlichkeitstheorie und Verw.\ Gebiete\/} {\bf 14} (1969), 89--92.

\bibitem{kakutani_50} S.\ Kakutani, Random ergodic theorems and Markoff processes with a stable
distribution. {\it Proceedings of the Second Berkeley Symposium on Mathematical Statistics
and Probability}, 1950, 247--261. University of California Press, Berkeley and Los Angeles,  1951. 

\bibitem{Kestelman_47}
H.\ Kestelman, On the functional equation $f(x+y)=f(x)+f(y)$, {\it Fund.\ Math.} {\bf 34}
(1947), 144--147.

\bibitem{krengel} U.\ Krengel, {\em Ergodic theorems}, de Gruyter, Berlin, 1985.

\bibitem{KN} L.\ Kuipers and H.\ Niederreiter, {\em Uniform distribution of
sequences}, John Wiley \& Sons, New York--London--Sidney, 1974.

\bibitem{loynes_73} R.M.\ Loynes, Some results in the probabilistic theory of
asymptotic uniform distribution modulo 1, {\it Z.\ Wahrscheinlichkeitstheorie 
und Verw.\ Gebiete\/} {\bf 26} (1973), 33--41.

\bibitem{Meyn_Tweedie_93}
S.\ Meyn and R.L.\ Tweedie, {\em Markov Chains and Stochastic Stability} (2nd edn.),
Cambridge University Press, Cambridge, 2009.

\bibitem{robbins_53} H.\ Robbins, On the equidistribution of sums of independent
random variables, {\it Proc.\ Amer.\ Math.\ Soc.} {\bf 4} (1953), 786--799.

\bibitem{Rogers_Williams_94}
L.C.G.\ Rogers and D.\ Williams, {\em Diffusions, Markov processes, and martingales. Vol.\ 1},
John Wiley \& Sons, Chichester, 1994.

\bibitem{ross_10} K.A.\ Ross, Benford's Law, a growth industry, to appear in {\em Amer.\ Math.\ Monthly} (2011).

\bibitem{Ryll-Nardzewski_54}
C.\ Ryll-Nardzewski, {On the ergodic theorems (III). The random ergodic theorem},
{\em Studia Math.} {\bf 14} (1954), 298--301.

\bibitem{schatte_85} P.\ Schatte, The asymptotic uniform distribution modulo 1 of cumulative
processes, {\it Optimization\/} {\bf 16} (1985), 783--786.

\bibitem{schuerger_08} K.\ Sch\"{u}rger, Extensions of Black--Scholes processes and
Benford's law, {\it Stochastic Processes App.} {\bf 118} (2008), 1219--1243.

\bibitem{schuerger_11} K.\ Sch\"{u}rger, L\'evy processes and Benford's Law,
{\em preprint\/} (2011).

\bibitem{stack_71} O.\ Stackelberg, A uniform law of the iterated logarithm for
functions $C$-uniformly distributed mod 1, {\em Indiana Univ.\ Math.\ J.} {\bf 21}
(1971), 515--528.

\bibitem{stadje_89} W.\ Stadje, On the asymptotic equidistribution of sums of
independent identically distributed random variables, {\it Ann.\ Inst.\ H.\ Poincar\'e Probab.\
Statist.} {\bf 25} (1989), 195--203.

\bibitem{Steinhaus_20}
H.\ Steinhaus, Sur les distances des points dans les ensembles de mesure positive,
{\it Fund.\ Math.} {\bf 1} (1920), 93--104.

\bibitem{uvn} S.M.\ Ulam and J.\ von Neumann, Random ergodic theorem, {\em Bull.\ Amer.\ Math.\ Soc.}
{\bf 51} (1945), 660.

\bibitem{walters_82} P.\ Walters, {\em An Introduction to Ergodic Theory}, Springer, New York--Heidelberg--Berlin, 1982.

\end{thebibliography}
\end{document}